\documentclass[12pt]{article}
\usepackage{amssymb}
\usepackage{amsmath,amsthm}
\usepackage[latin1]{inputenc}
\usepackage[dvips]{graphicx}
\usepackage{hyperref}
\usepackage{color}
\usepackage{cite}
\usepackage{enumerate}
\usepackage{tikz}
\usepackage{xifthen}

\hypersetup{colorlinks=true}

\hypersetup{colorlinks=true, linkcolor=blue, citecolor=blue,urlcolor=blue}

\newtheorem{remark}{Remark}[section]
\newtheorem{definition}{Definition}[section]

\newtheorem{lemma}[remark]{Lemma}
\newtheorem{theorem}[remark]{Theorem}

\newtheorem{example}[definition]{Example}
\newtheorem{corollary}[remark]{Corollary}

\title{On the $k$-metric Dimension of Metric Spaces}
\author{ and }

\author{A. F. Beardon$^{(1)}$ and J. A. Rodr\'{\i}guez-Vel\'{a}zquez$^{(2)}$\\$^{(1)}${\small Centre for Mathematical Sciences}\\{\small University of Cambridge}\\{\small Wilberforce Road, Cambridge CB3 0WB, United Kingdom}\\
   {\small
afb\@@dpmms.cam.ac.uk}\\
$^{(2)}${\small Departament d'Enginyeria Inform\`atica i Matem\`atiques }\\
{\small Universitat Rovira i Virgili } \\  {\small Av. Pa\"{\i}sos
Catalans 26, 43007 Tarragona, Spain.} \\{\small
   juanalberto.rodriguez\@@urv.cat}\\   
}

\begin{document}
\maketitle

\begin{abstract}
The metric dimension of a general metric space was defined in 1953,
applied to the set of vertices of a graph metric in 1975, and 
developed further for metric spaces in 2013. It was then generalised
in 2015 to the $k$-metric dimension of a graph for each positive
integer $k$, where $k=1$ corresponds to the original definition. Here,
we discuss the $k$-metric dimension of general metric spaces.
\end{abstract}


\section{Introduction}
The metric dimension of a general metric space was introduced in 1953
in \cite[p.95]{Blu53} but attracted little attention until, about
twenty years later, it was applied to the distances between vertices
of a graph \cite{Harary1976, Khuller1996, Mel84-113,Slater1975}. Since
then it has been frequently used in graph theory, chemistry, biology,
robotics and many other disciplines. The theory was developed further
in 2013 for general metric spaces \cite{Sheng2013}. More recently,
the theory of metric dimension has been generalised, again in the
context of graph theory, to the notion of a $k$-metric dimension,
where $k$ is any positive integer, and where the case $k=1$
corresponds to the original theory 
\cite{Estrada-Moreno2013,Estrada-Moreno2013corona,Estrada-Moreno2014a,
Estrada-Moreno2014b,Estrada-Moreno2014}.
Here we develop the idea of the $k$-metric dimension both in graph
theory and in metric spaces. As the theory is trivial when the space
has at most two points, we shall assume that any space we are
considering has \emph{at least three points}. 
Finally, whenever we discuss a connected graph $G$, we shall always 
consider the metric space $(X,d)$, where $X$ is the vertex set of
$G$, and $d$ is the usual graph metric in which the distance between
two vertices is the smallest number of edges that connect them.

Let $(X,d)$ be a metric space. If $X$ is a finite set, we denote its
cardinality by $|X|$; if $X$ is an infinite set, we put $|X| =
+\infty$. In fact, it is possible to develop the theory with $|X|$ any
cardinal number, but we shall not do this. The distances from a point
$x$ in $X$ to the points $a$ in a subset $A$ of $X$ are given by the
function $a \mapsto d(x,a)$, and the subset $A$ is said to
\emph{resolve} $X$ if each point $x$ is uniquely determined by this
function.  Thus $A$ resolves $X$ if and only if $d(x,a)=d(y,a)$ 
for all $a$ in $A$ implies that $x=y$; informally, if an object in
$x$ knows its distance from each point of $A$, then it knows exactly
where it is located in $X$. The class $\mathcal{R}(X)$ of subsets of
$X$ that resolve $X$ is non-empty since $X$ resolves $X$. The
\emph{metric dimension} ${\rm dim}(X)$ of $(X,d)$ is minimum value of
$|S|$ taken over all $S$ in $\mathcal{R}(X)$. 
The sets in $\mathcal{R}(X)$ are
called the \emph{metric generators}, or \emph{resolving subsets}, of
$X$, and $S$ is a \emph{metric basis} of $X$ if $S\in \mathcal{R}(X)$
and $|S|= {\rm dim}(X)$. A metric generator of a metric space $(X,d)$
is, in effect, a \emph{global co-ordinate system} on $X$. For example,
if $(x_1,\ldots,x_m)$ is an ordered metric generator of $X$, then the
map $\Delta:X \to \mathbb{R}^m$ given by
\begin{equation}\label{151123a}
\Delta(x)= \Big(d(x,x_1), \ldots,d(x,x_m)\Big)
\end{equation}
is injective (for this vector determines $x$), so that $\Delta$ is a
bijection from $X$ to a subset of $\mathbb{R}^m$, and $X$ inherits its
co-ordinates from this subset.

Now let $k$ be a positive integer, and $(X,d)$ a metric space. A
subset $S$ of $X$ is a \emph{$k$-metric generator} for $X$ (see
\cite{Estrada-Moreno2013}) if and only if any pair of points in $X$ is
distinguished by at least $k$ elements of $S$: that is, for any pair
of distinct points $u$ and $v$ in $X$, there exist $k$ points
$w_1,w_2,...,w_k$ in $S$ such that $$d(u,w_i)\ne d(v,w_i),\quad
i=1,\ldots,k.$$

A $k$-metric generator of minimum cardinality in $X$ is called a
\emph{$k$-metric basis}, and its cardinality, which is denoted by
$\dim_{k}(X)$, is called the $k$-\emph{metric dimension} of $X$.
 Let $\mathcal{R}_k(X)$ be the set of $k$-metric generators for $X$.
Since $\mathcal{R}_1(X)=\mathcal{R}(X)$,
we see that $\dim_1(X) = \dim(X)$.
 Also, as $\inf\varnothing =
+\infty$, this means that $\dim_k(X) = +\infty$ if and only if no
finite subset of $X$ is a $k$-metric generator for $X$.

Given a metric space $(X,d)$, we define the \emph{dimension sequence}
of $X$ to be the sequence 
$$ \big({\rm dim}_1(X), {\rm dim}_2(X), \ldots, {\rm dim}_k(X),
\ldots\big),$$
and we address the following two problems. 
\begin{itemize}
\item 
Can we find necessary and sufficient conditions for a sequence
$(d_1,d_2,d_3,\ldots)$ to be the dimension sequence of some metric
space?
\item
How does the dimension sequence of $(X,d)$ relate to the properties of
$(X,d)$? 
\end{itemize}

In Sections \ref{bisectors}, \ref{genresults} and \ref{coordinates}
we provide some basic results on the $k$-metric dimension, and in
Section \ref{someexamples} we calculate the dimension sequences of
some metric spaces. We then apply these ideas to the join of two
metric spaces, and to the Cayley graph of a finitely generated group.


\section{Bisectors}\label{bisectors}
As shown in \cite{Sheng2013}, the ideas about metric dimension are
best described in terms of bisectors. For \emph{distinct} $u$ and $v$
in $X$, the \emph{bisector} $B(u|v)$ of $u$ and $v$ is given by
$$B(u|v) = \{x\in X: d(x,u)=d(x,v)\}.$$
The complement of $B(u|v)$ is denoted by $B^c(u|v)$; thus 
$$B^c(u|v) = \{x \in X: d(x,u) \neq d(x,v)\},$$ and this contains
both $u$ and $v$. Whenever we speak of a bisector $B$, we shall assume
that it is some bisector $B(u|v)$, where $u\neq v$, so that its
complement $B^c$ is not empty.

Let us now consider the $k$-metric dimension from the perspective of
bisectors. A subset $A$ of $X$ \emph{fails to resolve} $X$ if and only
if there are distinct points $u$ and $v$ in $X$ such that $d(u,a) =
d(v,a)$ for all $a$ in $A$. Thus $A$ resolves $X$ if and only if $A$
is not contained in any bisector or, equivalently, \emph{if and only
if for every bisector $B$, we have} $|B^c \cap A| \geq 1$. This leads
to an alternative (but equivalent) definition of the metric dimension
${\rm dim}(X)$, namely  
$${\rm dim}(X) = \inf\{|A|: A\subset X \ 
\text{and, for all bisectors\ }B,\  |B^c\cap A| \geq 1\}.$$
Again, this infimum may be $+\infty$. The extension to the $k$-metric
dimension ${\rm dim}_k(X)$ of $X$ is straightforward:
\begin{equation}\label{151123c}
{\rm dim}_k(X) = \inf\{|A|: A\subset X \  
\text{and, for all bisectors\ }B,\  |B^c\cap A| \geq k\}.
\end{equation} 
Note that if $X$ is a finite set then ${\rm dim}_{|X|+1}(X) =
+\infty$. 

Clearly, the values ${\rm dim}_k(X)$ depend only on the class
$\mathcal{B}$ of bisectors in $X$; for example, ${\rm dim}_1(X)=1$ if
and only if there is some point in $X$ that is not in any bisector.
More generally, in all cases, $\dim_k(X) \geq k$, and equality
holds here if and only if there are $k$ points of $X$ that do not lie
in any bisector. For example, if $X$ is the real, closed interval
$[0,1]$ with the Euclidean metric, then ${\rm dim}_k(X) = k$ for
$k=1,2$. For a more general example of this type, let $X =
\{\sqrt{p}: \text{$p$ \ a prime number}\}$ with the Euclidean metric
$d$. If $p$, $q$ and $r$ are primes, with $p\neq q$, then $\sqrt{r}
\in B(\sqrt{p}|\sqrt{q})$ implies $\sqrt{r} =
\tfrac12(\sqrt{p}+\sqrt{q})$;
hence $4r= p+q+2\sqrt{pq}$. Since $\sqrt{pq}$ is irrational, this is
false; hence every bisector is empty. It follows that ${\rm dim}_k(X)
= k$ for $k=1,2,\ldots$; thus the dimension sequence of $(X,d)$ is
$(1,2,3,\ldots)$.


\section{The monotonicity of dimensions}\label{genresults}
Let $(X,d)$ be a metric space. Then, from \eqref{151123c}, we have
${\rm dim}_k(X) \leq {\rm dim}_{k+1}(X)$, but we shall now establish
the stronger inequality ${\rm dim}_k(X) +1 \leq  {\rm dim}_{k+1}(X)$
(which is ${\rm dim}_k(X) < {\rm dim}_{k+1}(X)$ when the dimensions
are finite, but not when they are $+\infty$). This inequality
is known for graphs; see
\cite{Estrada-Moreno2013,Estrada-Moreno2013corona})
where it is an important tool. 

\begin{theorem}\label{Monoty-AFB}
Let $(X,d)$ be a metric space. Then, for $k = 1,2,\ldots$, 
\\ {\rm (i)} 
if ${\rm dim}_k(X) < +\infty$ then $\dim_k(X) < \dim_{k+1}(X)$; 
\\ {\rm (ii)} 
if ${\rm dim}_k(X) = +\infty$ then $\dim_{k+1}(X) = +\infty$.
\\ In particular, $\dim_k(X) +1 \geq \dim_1(X)+ k$.
\end{theorem}

\begin{proof} 
First, (ii) follows immediately from \eqref{151123c}. Next, (i) is
true if $\dim_{k+1}(X) = +\infty$, so we may assume that
$\dim_{k+1}(X) = p < +\infty$. Thus there is a subset
$\{x_1,\ldots,x_p\}$ (with the $x_i$ distinct) of $X$ such that for
every bisector $B$, $|B^c \cap\{x_1,\ldots,x_p\}| \geq k+1$. As
$k \geq 1$ we see that $p \geq 2$. 
Clearly, $|B^c \cap \{x_1,\ldots,x_{p-1}\}| \geq k$ for every bisector
$B$; hence $\dim_k(X) \leq p-1 < \dim_{k+1}(X)$. The last inequality
follows by induction.
\end{proof}


\section{The $1$-metric dimension}\label{coordinates}

Theorem \ref{Monoty-AFB} shows that if $+\infty$ occurs as a term in
the dimension sequence of $(X,d)$, then all subsequent terms are also
$+\infty$. Thus $\dim_1(X) = +\infty$ if and only if $(X,d)$ has
dimension sequence $(+\infty,+\infty,+\infty,\ldots)$. The next result
shows when this is so. 

\begin{theorem}\label{InfinityDimension}
Let $(X,d)$ be a metric space. Then $\dim_1(X) = +\infty$ if and only
if every finite subset of $X$ lies in some bisector. In particular,
if $X$ is the union of an increasing sequence of bisectors, then
$\dim_1(X) = +\infty$.
\end{theorem}

\begin{proof}
First, the definition of ${\rm dim}(X)$ implies that ${\rm dim}_1(X) =
+\infty$ if and only if every finite subset of $X$ lies in some
bisector. The second statement holds because if $X = \cup_n B_n$,
where $B_1,B_2,\ldots$ is an increasing sequence of bisectors, then,
given any finite subset $\{x_1,\ldots,x_m\}$ of
$X$, each $x_j$ lies in some $B_{i_j}$, and $\{x_1,\ldots,x_m\}
\subset B_r$, where $r = \max\{i_1,\ldots,i_m\}$.
\end{proof}

What can be said if ${\rm dim}_1(X)< + \infty$? It seems that we can
obtain very little information from the \emph{single} assumption that
$\dim_1(X) < +\infty$; for example, for each $r \geq 0$ choose a point
$x_r$ in $\mathbb{R}^n$ with $\|x_r\|=r$, and let $X = \{x_r: r \geq
0\}$. Then $\{0\}$ is a $1$-metric basis for $X$, and ${\rm
dim}_1(X)=1$ but we can say almost nothing about the topological
structure of $X$. However, we can say more if we know that $X$ is
compact.

\begin{theorem}\label{151123d}
Let $(X,d)$ be a compact metric space with ${\rm dim}_1(X)=m <
+\infty$. Then $(X,d)$ is homeomorphic to a compact subset of
$\mathbb{R}^m$. 
\end{theorem}

\begin{proof}
Suppose that $X$ is compact, and that ${\rm dim}_1(X) = m < +\infty$.
Then there is a $1$-metric basis $\{x_1,\ldots,x_m\}$, and the
corresponding bijection $\Delta$ in \eqref{151123a} that maps $X$ onto
some subset of $\mathbb{R}^m$. Now $\Delta$ is
continuous on $X$ since, for each $j$, we have
$$|\Delta(x)-\Delta(y)| \leq \sum_{j=1}^m |d(x,x_j)-d(y,x_j)| \leq
m d(x,y).$$ 
As $\Delta$ is a continuous, injective map from a compact space to the
Hausdorff space $\mathbb{R}^m$ it follows (by a well known result in
topology) that it is a homeomorphism. 
\end{proof}

This result is related to the following result in \cite{Sheng2013}
(see also \cite{Murphy}).

\begin{theorem} \label{160128a}
If $(X,d)$ is a compact, connected metric space with $\dim_1(X)=1$
then $X$ is homeomorphic to $[0,1]$.
\end{theorem}

The compactness is essential here as there is an example in
\cite{Sheng2013} of a connected, but not arcwise connected, metric
space $X$ with ${\rm dim}_1(X)=1$. As $X$ is not arcwise connected, it
is not homeomorphic to $[0,1]$. It is conjectured in \cite{Sheng2013}
that if $X$ is arcwise connected, and ${\rm dim}_1(X)=1$ then $X$ is a
Jordan arc (this means that $X$ is homeomorphic to one of the real
intervals $[0,1]$ and $[0,+\infty)$), and we can now show that this is
so.

\begin{theorem}
If $X$ is an arcwise connected metric space with ${\rm dim}_1(X)=1$,
then $X$ is a Jordan arc.
\end{theorem}

\begin{proof}
As ${\rm dim}_1(X)=1$, there is a metric basis, say $\{x_0\}$ for $X$,
and every point $x$ of $X$ is uniquely determined by its distance
$d(x,x_0)$ from $x_0$. Consider the map $\Delta:x \mapsto d(x,x_0)$ of
$X$ into $[0,+\infty)$. This map is (uniformly) continuous because
$$|\Delta(x)-\Delta(y)| = |d(x,x_0)-d(y,x_0)| \leq d(x,y),$$
and as $X$ is arcwise connected (and therefore connected), so
$\Delta(X)$ is connected. This means that $\Delta$ is an interval of
the form $[0,a]$, where $a>0$, or $[0,b)$, where $0 < b \leq +\infty$.

Let us consider the case when $\Delta(X)= [0,a]$. As $\Delta$ is
injective, we see that for every $r$ in the interval
$[0,a]$ there is some unique $x_r$ in $X$ with $d(x_r,x_0)=r$.
Thus $X = \{x_r:0 \leq r \leq a\}$. However, as $X$ is arcwise
connected, there is a curve, say $\gamma: [0,1]\to X$ with
$\gamma(0)=x_0$ and $\gamma(1)=x_a$. Now as $\gamma$ is continuous,
the set $\{d\big(\gamma(t),x_0\big): t\in [0,1]\}$ must
contain every real number in the interval $[0,a]$, and it cannot
contain any other numbers; thus $X = \gamma([0,1])$. Now
$\gamma([0,1])$ is compact for it is the continuous image of the
compact interval $[0,1]$; thus $X$ is compact and so, by Theorem
\ref{160128a}, $X$ is a Jordan arc.
 
The argument in the case when $\Delta(X) = [0,b)$ is similar. Indeed,
the argument above holds for every $a$ with $0<a<b$, and it is easy
to see that this implies that $\Delta$ is a homeomorphism from
$X$ to $[0,b)$.
\end{proof}


\section{Some examples}\label{someexamples}
In order to calculate the $k$-metric dimension of a metric space we
need to understand the geometric structure of its bisectors, and we
now illustrate this with several examples. In order to maintain
the flow of ideas, the details of these examples will be given later.

\begin{example} \label{151006a} {\rm 
Let $(X,d)$ be any one of the Euclidean, spherical and hyperbolic
spaces $\mathbb{R}^n$, $\mathbb{S}^n$ and $\mathbb{H}^n$,
respectively, each with the standard metric of constant curvature $0$,
$1$ and $-1$, respectively. The bisectors are well understood in 
these spaces, and we shall show that any non-empty open subset of $X$
has $k$-metric dimension $n+k$. In particular, each of these spaces
has dimension sequence $(n+1,n+2,n+3,\ldots)$. See
\cite{Hey14-230,Sheng2013} for the $1$-metric dimensions of these
spaces. 
} \end{example}

\begin{example}\label{exampleC}{\rm 
Let $X$ be any finite set with the discrete metric $d$ (equivalently,
$X$ is the vertex set of a complete, finite graph). For distinct $u$
and $v$ in $X$ we have $B(u|v)= X \backslash\{u,v\}$, so that for any
subset $S$ of $X$, we have
$B(u|v)^c \cap S = \{u,v\} \cap S$. Thus if $|S \cap B^c|\geq 1$ for
all bisectors $B$, then $S$ can omit at most one point of $X$. We
conclude that ${\rm dim}_1(X) = |X|-1$. If $|B^c\cap S| \geq 2$ for
all bisectors $B$ then $S=X$, and ${\rm dim}_2(X) = |X|$. We conclude
that $(X,d)$ has dimension sequence
$(|X|-1,|X|,+\infty,+\infty,\ldots)$.
}\end{example}

\begin{example}\label{exampleP} {\rm
Let $X$ be the real interval $[0,1]$, with the Euclidean metric.
Then $B$ is a bisectors if and only if $B = \{x\}$ for some $x$ in
$(0,1)$. Thus $\{0\}$ is a $1$-metric basis, and $\{0,1\}$ is a
$2$-metric basis, of $[0,1]$. We leave the reader to show that if $k
\geq 3$ then $\{0, \tfrac1k,\tfrac2k,\ldots,\tfrac{k-1}k, 1\}$
is a $k$-metric basis, so that $[0,1]$ has dimension sequence 
$(1,2,4,5,6,\ldots)$. A similar argument shows that $[0,+\infty)$ has
dimension sequence $(1,3,4,5,\ldots)$, and that $(-\infty,+\infty)$,
which is $\mathbb{R}$, has dimension sequence $(2,3,4,\ldots)$.
}\end{example}

\begin{example}\label{exampleE} {\rm 
The  Petersen graph, which is illustrated in Figure 1, has dimension
sequence $(3,4,7,8,9,10,+\infty,\ldots)$. The (finite) values
${\rm dim}_k(X)$ for $k=1,\ldots,6$ come from a computer search, and
as ${\rm dim}_6(X) = 10 = |X|$, we have ${\rm dim}_7(X) = +\infty$.  
}\end{example}

\begin{figure}[!ht]
\centering
\begin{tikzpicture}[transform shape, inner sep = .7mm]
\def\radius{0.8} 
\foreach \ind in {1,...,5}
{
\pgfmathparse{-54+360/5*\ind};
\node [draw=black, shape=circle, fill=white] (v\ind) at
(\pgfmathresult:\radius cm) {};
\node [draw=black, shape=circle, fill=white] (u\ind) at
(\pgfmathresult:\radius+2 cm) {};
}
\foreach \ind in {1,...,5}
 {\draw[black] (u\ind) -- (v\ind);}
\foreach \ind in {1,...,4}
  {
 \pgfmathparse{int(\ind+1)}
 \draw[black] (u\ind) -- (u\pgfmathresult);
}  \draw[black] (u5) -- (u1); \draw[black] (v1) --
(v3)--(v5)--(v2)--(v4)--(v1);
\end{tikzpicture}
\caption{The Petersen graph.}\label{figPetersen}
\end{figure}
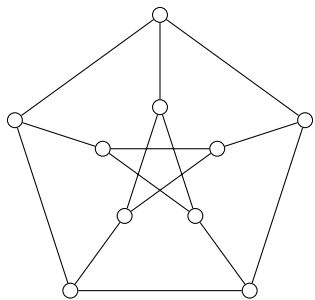

\begin{example}\label{exampleD}{\rm 
Let $G$ be a group with a given set of generators, let $V$ be the
vertex set of the associated Cayley graph of $G$, and let $d$ be its
graph metric.
\\ {\rm (i)}
If $G$ is an infinite cyclic group then $(V,d)$ has dimension sequence
$(2,3,4,\ldots)$.
\\ {\rm (ii)} 
If $G$ is a free group on $p$ generators, where $p\geq 2$, then
$(V,d)$ has dimension sequence $(+\infty,+\infty,+\infty,\ldots)$.
\\ {\rm (iii)} 
Let $G$ be an abelian group on $p$ generators, where $p\geq 2$,
and where each generator has infinite order. Then $(V,d)$ has
dimension sequence $(+\infty,+\infty,+\infty,\ldots)$.
}\end{example}


\section{Three geometries of constant curvature}
\label{SpacesOfConstantCurvature}
In this section we give the details of Example \ref{151006a}. 
It is shown in \cite{Sheng2013} that if $U$ is any non-empty, open
subset of any one of the three classical geometries $\mathbb{R}^n$,
$\mathbb{S}^n$ and $\mathbb{H}^n$, then $\dim_1(U)=n+1$. Here we show
that if $X$ is any of these spaces then $\dim_k(X)=n+k$ for
$k=1,2,\ldots$. The same result holds for non-empty \emph{open}
subsets of these spaces, and we leave the reader to make the
appropriate changes to the proofs. 

The proof that $\dim_k(X)=n+k$ when $X$ is one of the three geometries
$\mathbb{R}^n$, $\mathbb{S}^n$ and $\mathbb{H}^n$, is largely
independent of the choice of $X$, and depends only on the nature of
the bisectors in these geometries. Each of these three geometries has
the following properties:
\\ (P1) \quad ${\rm dim}_1(X) = n+1$;
\\ (P2) \quad there exists $x_1,x_2,\ldots$ in $X$ such that if
$j_1 < j_2< \cdots < j_n$ then $\{x_{j_1},\ldots,x_{j_n}\}$ lies 
on a unique bisector $B$, and no other $x_i$ lies on $B$.
\\ Now (P1) and (P2) imply that $\dim_k(X)=n+k$ for
$k=1,2,\ldots$. Indeed, (P2) implies that for any bisector $B$,
$|B \cap \{x_1,\ldots,x_{n+k}\}| \leq n$, so that
$|B^c \cap \{x_1,\ldots,x_{n+k}\}| \geq k$. 
This implies that ${\rm dim}_k(X) \leq n+k$. However, (P1) and Theorem
\ref{Monoty-AFB} show that ${\rm dim}_k(X) \geq n+k$. Since we know
that  each of $\mathbb{R}^n$, $\mathbb{S}^n$ and $\mathbb{H}^n$ has
the property (P1), it remains to show that they have the property
(P2), and this depends on the nature of the bisectors in these
geometries. We consider each in turn.

\vspace{3mm} \noindent {\bf Euclidean Space $\mathbb{R}^n$}\\
Each bisector in $\mathbb{R}^n$ is a hyperplane (that is, the
translation of an $(n-1)$-dimensional subspace of $\mathbb{R}^n$),
and each hyperplane is a bisector. Any set of $n$ points lies on a
bisector, and there exists sets of $n+1$ points that do not lie on 
any single bisector. The appropriate geometry here is the affine
geometry of $\mathbb{R}^n$, but we shall take a more informal view.
First, we choose $n$ points $x_1,\ldots,x_n$ that lie on a unique
hyperplane $H$. Next, we select a point $x_{n+1}$ not on $H$. Then
any $n$ points chosen from $\{x_1,\ldots,x_{n+1}\}$ lie on some
hyperplane $H'$, and the remaining point does not lie on $H'$. Now
suppose that we have constructed the set $\{x_1,\ldots,x_{n+p}\}$
with the property that any set of $n$ points chosen from this lie
on a unique hyperplane, say $H_\alpha$, and that no other $x_i$ lies
on $H_\alpha$. Then we can choose a point $x_{n+p+1}$ that is not not
on any of the $\binom{n+p}{n}$ hyperplanes $H_\alpha$, and it is then
easy to check that the sequence $x_1,x_2,\ldots$ has the property
(P2).

Although we have not used it, we mention that there is a formula for
the $n$-dimensional volume $V$ of the Euclidean simplex whose vertices
are the $n+1$ points $x_1,\ldots,x_{n+1}$ in $\mathbb{R}^n$, namely
$$V^2 = \frac{(-1)^{n+1}}{2^n(n\,!)^2}\,\Delta,$$
where $\Delta$ is the \emph{Cayley-Menger determinant} given by
$$ \Delta = \left |
\begin{matrix}
0 &1 &\cdots &1 \\
1 &d_{1,1}^2 &\cdots &d_{1,n+1}^2\\
\vdots &\vdots &\ddots &\vdots \\
1 &d_{n+1,1}^2 &\cdots &d_{n+1,n+1}^2\\
\end{matrix}
\right |,$$
and $d_{i,j} = \|x_i-x_j\|$. As $V=0$ precisely when the points $x_j$
lie on a hyperplane, we see that this condition could be used to
provide
an algebraic background to the discussion above. For more details,
see \cite{Ber87}, \cite{Blu53} and \cite{Blu70}.
We also mention that there are versions of the Cayley-Menger
determinant that are applicable to spherical, and to hyperbolic,
spaces.

\vspace{3mm} \noindent {\bf Spherical Space $\mathbb{S}^n$}\\ 
Spherical space $(\mathbb{S}^n,d)$ is the space 
$\{x \in \mathbb{R}^{n+1} :\; \| x \| = 1\}$ with the path metric
$d$ induced on $\mathbb{S}^n$ by the Euclidean metric on
$\mathbb{R}^{n+1}$. Explicitly, $\cos d(x,y) = x{\cdot}y$, where 
$x{\cdot}y$ is the usual scalar product in $\mathbb{R}^{n+1}$. 
If $u$ and $v$ are distinct points of $\mathbb{S}^n$, we let
$B^{\mathcal{E}}(u|v)$ be the \emph{Euclidean} bisector (in
$\mathbb{R}^{n+1}$) of $u$ and $v$, and $B^{\mathcal{S}}(u|v)$ the 
spherical bisector in the space $(\mathbb{S}^n,d)$. Then
$B^{\mathcal{E}}(u|v)$ is a hyperplane that passes through the origin
in $\mathbb{R}^{n+1}$, and 
\begin{equation}\label{151013c}
B^{\mathcal{S}}(u|v) = \mathbb{S}^n\cap B^{\mathcal{E}}(u|v).
\end{equation}
The bisectors $B^{\mathcal{S}}(u|v)$ are the \emph{great circles} (of
the appropriate dimension) on $\mathbb{S}^n$.

The equation \eqref{151013c} implies that the $k$-metric dimension of
the spherical spaces is the same as for Euclidean spaces. Indeed, our
proof for Euclidean spaces depended on constructing a sequence
$x_1,x_2,\ldots$ with the property (P2), and it is clear that this
construction could be carried out in such a way that each $x_j$ lies
on $\mathbb{S}^n$.

\vspace{3mm} \noindent {\bf Hyperbolic Space $\mathbb{H}^n$}\\ 
Our model of hyperbolic $n$-dimensional space is Poincare's half-space
model 
$$\mathbb{H}^n =
\{(x_1,\ldots,x_{n+1}\in \mathbb{R}^{n+1}:\; x_{n+1}>0\}$$ 
equipped with the hyperbolic distance $d$ which is derived from
Riemannian metric ${|dx|}/{x_{n+1}}$. For more details, see for
example, \cite{MR698777,MR2249478}. Our argument for $\mathbb{H}^n$ 
is essentially the same as for $\mathbb{R}^n$ and $\mathbb{S}^n$
because if $u$ and $v$ are distinct points in $\mathbb{H}^n$, then
the hyperbolic bisector $B(u|v)$ is the set $S \cap \mathbb{H}^n$,
where $S$ is some Euclidean sphere whose centre lies on the hyperplane
$x_n = 0$. We omit the details.


\section{The metric dimensions of graphs}
The vertex set $V$ of a graph $G$ supports a natural \emph{graph
metric} $d$, where $d(u,v)$ is the smallest number of edges that can
be used to join $u$ to $v$. Some basic results on the $k$-metric
dimension of a graph have recently been obtained in
\cite{Estrada-Moreno2013,Estrada-Moreno2013corona,Estrada-Moreno2014b,
Estrada-Moreno2014,Estrada-Moreno2014a}. Moreover, it was shown in
\cite{Yero2013c} that the problem of computing the $k$-metric
dimension of a graph is NP-complete. 
A natural problem in the study of the $k$-metric dimension of a metric
space $(X,d)$ consists of finding the largest integer $k$ such that
there exists a $k$-metric generator for $X$. For instance, for the
graph shown in Figure \ref{figDimk} the maximum value of $k$ is four. 
It was shown in \cite{Estrada-Moreno2013corona,Estrada-Moreno2014a}
that for any graph of order $n$ this problem has time complexity
of order ${O}(n^3)$. If we consider the discrete metric space 
$(X,d_0)$ (equivalently, a compete graph), then $\dim_1(X)=|X|-1$ and
$\dim_2(X)=|X|$.
Furthermore, for $k\ge 3$ there are no $k$-metric generators for $X$.
In general, for any metric space $(X,d)$, the whole space $X$ is a
$2$-metric generator, as two vertices are distinguished by themselves.
As we have already seen, there are metric spaces, like the Euclidean
space $\mathbb{R}^n$, where for any positive integer $k$, there exist
at least one $k$-metric generator. 

We shall now discuss the dimension sequences of the simplest connected graphs,
namely paths and cycles (and we omit the elementary details).

A finite \emph{path} $P_n$ is a graph with vertices $v_1,\ldots,v_n$,
edges $[v_1,v_2],\ldots,[v_{n-1},v_n]$, and bisectors
$\{v_2\},\ldots,\{v_{n-1}\}$. We leave the reader to show that $P_n$
has dimension sequence  
$$
\begin{cases}
(1,2,+\infty,\ldots) &\text{if $n=2,3$;}\\
(1,2,4,5,\ldots,n,+\infty,\ldots) &\text{if $n \geq 4$.}
\end{cases}$$
A \emph{semi-infinite path} $P_{\mathbb{N}}$ is a graph with vertices
$v_1,v_2\ldots$, edges $[v_1,v_2], [v_2,v_3],\ldots$, and bisectors
$\{v_2\},\ldots$. Thus $P_{\mathbb{N}}$ has dimension sequence 
$(1,3,4,5,\ldots)$.
A \emph{doubly-infinite path} $P_{\mathbb{Z}}$ is the graph with
vertices $\ldots,v_{-1},v_0,v_1,\ldots$, edges
$\ldots,[v_{-1},v_0],[v_0,v_1],\ldots$, and bisectors 
$\ldots,\{v_{-1}\},\{v_0\},\{v_1\},\ldots$. Thus $P_{\mathbb{Z}}$
has dimension sequence $(2,3,4,5,\ldots)$. We note that a graph $G$
has $1$-metric dimension $1$ if and only it is $P_n$ or
$P_{\mathbb{N}}$ \cite{Khuller1996,Chartrand2000}. This, together with
the results just stated, show that if $G$ is a graph of order two or
more, and $k\geq 2$, then $\dim_k(G)=k$ if and only if $G$ is $P_n$
and $k=2$ (see also \cite{Estrada-Moreno2013}).

We now consider cycles. A cycle $C_n$ is a graph with vertices
$v_1,\ldots, v_n$, and edges $\{v_1,v_2\}, \ldots, \{v_{n-1},v_n\},
\{v_n,v_1\}$. We must distinguish between the cases where $n$ is even,
and where $n$ is odd (which is the easier of the two cases) and, as
typical examples, we mention that $C_7$ has dimension sequence
$(2,3,\ldots,7,+\infty,\ldots)$, and $C_8$ has
dimension sequence $(2,3,4,6,7,8,+\infty,\ldots)$. Suppose that $n$ is
odd; then the bisectors are the singletons $\{v\}$. Thus if $S$ is a
set of $k+1$ vertices, where $k+1 \leq n$, then $|B^c\cap S| \geq k$
for every bisector $B$. Thus if $n$ is odd, then ${\rm dim}_k(C_n) =
k+1$, and $C_n$ has dimension sequence
$(2,3,\ldots,n,+\infty,\ldots)$.

We now show that $C_{2q}$ has dimension sequence
$$(2,3,\ldots,q, q+2,q+3,\ldots,q+q,+\infty,\ldots).$$ 
To see this,
label the vertices as $v_j$, where $j\in \mathbb{Z}$, and where
$v_i = v_j$ if and only if $i\equiv j \pmod{n}$. The vertices
$v_i$ and $v_j$ are \emph{antipodal vertices} if and only if 
$i-j \equiv q \pmod{2q}$; thus $v_j$ and $v_{j+q}$ are antipodal
vertices.  The class of bisectors is the class of sets $\{v,v^*\}$,
where $v$ is a vertex, and $v^*$ is the vertex that is antipodal to
$v$. For $k = 1,\ldots,q-1$ we can take a set of $k+1$ points, no
two of which are antipodal, as a $k$-metric basis, so that ${\rm
dim}_k(C_{2q}) = k+1$ for $k=1,\ldots,q-1$. To find ${\rm
dim}_q(C_{2q})$, we need to take (for a $q$-metric basis)
a set $S$ which contains two pairs of antipodal points, and one more
point from each pair of the remining antipodal pairs. We leave the
details to the reader.

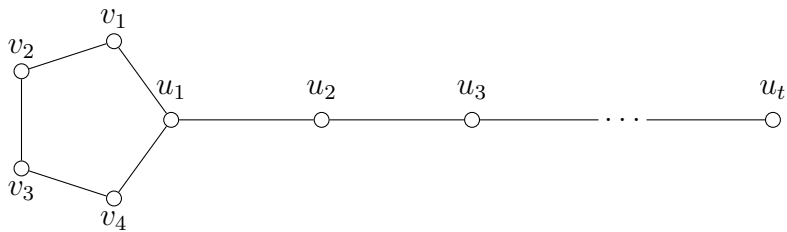
\begin{figure}[!ht]
\centering
\begin{tikzpicture}[transform shape, inner sep = .7mm]
\def\radius{1.1} 
\foreach \ind in {1,...,4}
{
\pgfmathparse{360/5*\ind};
\node [draw=black, shape=circle, fill=white] (v\ind) at
(\pgfmathresult:\radius cm) {};
\ifthenelse{\ind=3\OR \ind=4}
{
\node [scale=1] at ([yshift=-.3 cm]v\ind) {$v_\ind$};
}
{
\node [scale=1] at ([yshift=.3 cm]v\ind) {$v_\ind$};
};
}
\foreach \ind in {1,...,3}
\pgfmathparse{int(\ind+1)}
\draw[black] (v\ind) -- (v\pgfmathresult);
\foreach \ind in {1,...,3}
{
\pgfmathparse{\radius + 2*(\ind-1)};
\node [draw=black, shape=circle, fill=white] (u\ind) at
(\pgfmathresult cm, 0) {};
\node [scale=1] at ([yshift=.4 cm]u\ind) {$u_\ind$};
}
\pgfmathparse{\radius + 2*3};
\node (ldot) at (\pgfmathresult cm, 0) {$\ldots$};
\pgfmathparse{\radius + 2*4};
\node [draw=black, shape=circle, fill=white] (ut) at (\pgfmathresult
cm, 0) {};
\node [scale=1] at ([yshift=.4 cm]ut) {$u_t$};
\foreach \ind in {1,4}
\draw[black] (u1) -- (v\ind);
\foreach \ind in {1,2}
\pgfmathparse{int(\ind+1)}
\draw[black] (u\ind) -- (u\pgfmathresult);
\draw[black] (u3) -- (ldot);
\draw[black] (ldot) -- (ut);
\end{tikzpicture}
\caption{For $k\in \{1,2,3,4\}$, $\dim_k(G)=k+1$.}\label{figDimk}
\end{figure}

As an example which joins a path to a cycle, consider the graph $G$
illustrated in Figure
\ref{figDimk} which is obtained from the cycle graph $C_5$ and the
path $P_t$, by identifying one of the vertices of the cycle, say
$u_1$, and one of the end vertices of $P_t$. Let $S_1 =
\{v_1,v_2\}$, $S_2 = \{v_1,v_2,u_t\}$, $S_3 = \{v_1,v_2,v_3,u_t\}$ 
and $S_4 = \{v_1,v_2,v_3,v_4,u_t\}$. Then, for $k = 1,2,3,4$, the set
$S_k$ is $k$-metric basis of $G$.

The following lemma is useful when discussing examples in
graph theory.

\begin{lemma}\label{151123b}
Suppose that a graph $G$ does not have any cycles of odd length.
Then $B(u|v) = \varnothing$ when $d(u,v)$ is odd.
\end{lemma}

The proof is trivial for if $x\in B(u|v)$ then there is a cycle of
odd length (from $u$ to $x$, then to $v$, and then back to $u$).
This lemma applies, for example, to the usual grid (or graph) in
$\mathbb{R}^n$ whose vertex set is $\mathbb{Z}^n$. A \emph{bipartite
graph} is a graph $G$ whose vertex set $V$ splits into complementary
sets $V_1$ and $V_2$ such that each of the edges of $G$ join a
point of $V_1$ to a point of $V_2$. As a graph is bipartite
if and only if it has no cycles of an odd length, this lemma is
about bipartite graphs.

\begin{example}\label{tree}{\rm 
Let us now consider a graph $G$ that is an infinite tree in which
every vertex has degree at least three. Now let $v$ be any vertex,
select three edges from $v$, say $[v,a]$, $[v,b]$ and $[v,c]$. As $G$
is a tree, if we remove one edge the remaining graph is disconnected. 
Now let $G_c$ be the subgraph of $G$ that would be the component
containing $c$ if we were to remove the edge $[v,c]$ from $G$.
 It is clear that if $u$ is a vertex in $G_c$, then
$d(a,u)=d(b,u)$ since any path from $a$ (or $b$) to $u$ must pass
through the edge $[v,c]$. We conclude that $G_c \subset B(a|b)$.
It is now clear from Theorem \ref{InfinityDimension} that $G$ has
dimension sequence $(+\infty,+\infty,\ldots)$.  
}\end{example}

For the rest of this section we shall consider the Cayley graph of a
group with a given set of generators as a metric space. Let $G$ be a
group and let $G_0$ a set of generators of $G$. We shall always
assume that if $g\in G_0$ then $g^{-1}\in G_0$ also. Then the Cayley
graph of the pair $(G,G_0)$ is a graph whose vertex set is $G$, and
such that the pair $(g_1,g_2)$ is an edge if and only if $g_2=g_0g_1$
for some $g_0$ in $G_0$. Thus, for example, $P_{\mathbb{Z}}$ is the
Cayley graph of an infinite cyclic group (on one generator), and $C_n$
is the Cayley graph of an finite cyclic group (on one generator).
We shall always assume that the set $G_0$ of generators of $G$ is
finite; then the Cayley graph is locally finite (that is, each
vertex is the endpoint of only finitely many edges). Note also that
if a generator $g_0$ has order two then $g_0^{-1}=g_0$ so this
only provides one edge (not two edges) from each vertex. The
following result, which characterises Cayley graphs within the class
of all graphs, is well known. 

\begin{theorem}\label{Cayley}
A graph is a Cayley graph of a group $G$ if and only if it admits a
simply transitive action of $G$ by graph automorphisms.
\end{theorem}
 
Theorem \ref{Cayley} suggests that if we use the homogeneity implied
by this result there is a reasonable chance of finding the dimension
sequence of a Cayley graph. However, for a graph that is not the
Cayley graph of a group, it seems that we are reduced to finding its
metric dimensions by a case by case analysis.

We shall now verify the claims made in Example \ref{exampleD}. First,
suppose that $G$ is a free group on $p$ generators. Then the Cayley
graph of $G$ is a tree in which every vertex has degree $2p$; thus,
using Example \ref{tree}, we see that $G$ has dimension sequence
$(+\infty,+\infty,+\infty,\ldots)$.

Next, we consider an abelian group $G$ on two generators of infinite
order (the proof for $p$ generators is entirely similar). The Cayley
graph of $G$ has $\mathbb{Z}^2$ as its vertex set and (if we identify
the lattice point $(m,n)$ with the Gaussian integer $m+in$) edges
$[m+in,m+1+in]$ and $[m+in, m+i(n+1)]$, where $m,n \in \mathbb{Z}$.
It is (geometrically) clear that for any $m\in \mathbb{Z}$ we have,
with $\zeta = m+im$, $$B(\zeta +1|\zeta +i) \supset \{p+iq: p\ge m+1,
q\ge m+1\}.$$
It now follows from Theorem \ref{InfinityDimension} (by taking $|m|$
large and $m$ negative) that $G$ has dimension sequence
$(+\infty,+\infty,\ldots)$.

In contrast to Example \ref{exampleD} we have the following result for
the infinite dihedral group whose Cayley graph is an infinite
ladder; for example we can take the group generated by the two
Euclidean isometries which, in complex terms, are $z \mapsto z+1$ and
$z \mapsto \bar z$.

\begin{theorem}
The infinite dihedral group has dimension sequence
$(3,4,6,8,\ldots)$.
\end{theorem}

\begin{proof}
We may assume that (in complex terms) the vertices of the ladder
graph are the points $m+in$, where $m\in \mathbb{Z}$ and $n = 0,1$.
The key to computing the metric dimensions of the ladder graph is the
observation that 
$$B(0|1+i) = \{1,2,3,\ldots\} \cup \{i, i-1,i-2,\ldots \}.$$
Of course, similar bisectors arise at other pairs of similarly
located points; equivalently, each automorphism of the graph maps a
bisector to a bisector. All other bisectors are either empty or of
cardinality two. We claim that $\{0,1,i\}$ is a $1$-metric basis for
the graph so that ${\rm dim}_1(G) = 3$.  Next, it is easy to see that
$\{0,1,i,1+i\}$ is a $2$-metric basis for $X$ so that ${\rm dim}_2(X)
= 4$. The set $\{0,1,2,i,1+i,2+i\}$ is a $3$-metric basis so that
${\rm dim}_3(X)=6$. We leave the details, and the remainder of the
proof to the reader.
\end{proof}


\section{The join of metric spaces}
The $k$-metric dimension of the \emph{join} $G_1+G_2$ of two finite
graphs $G_1$ and $G_2$ was studied in \cite{Estrada-Moreno2014a}.
Let us briefly recall the notion of the join of two graphs $G_1$ and
$G_2$ with disjoint vertex sets $V_1$ and $V_2$, respectively. The
\emph{join} $G_1+G_2$ of $G_1$ and $G_2$ is the graph whose
vertex set is $V_1 \cup V_2$, and whose edges are the edges in $G_1$,
the edges in $G_2$, together with all edges obtained by joining each
point in $V_1$ to each point in $V_2$. Let $d_1$, $d_2$ and $d$ be
the graph metrics of $G_1$, $G_2$ and $G_1+G_2$, respectively; then
\begin{equation*}
d(u,v) = 
\begin{cases}
\min\{d_1(u,v),2\} &\text{if \ $u,v\in V_1$}; \\
\min\{d_2(u,v), 2\} &\text{if \ $u,v\in V_2$};\\ 
1 &\text{if $u\in X_i$, $v\in X_j$, where $i\neq j$},
\end{cases}
\end{equation*} 
because if $u,v\in V_1$, say, then for $w$ in $V_2$, we have
$d_1(u,v) \leq d(u,w)+d(w,v)=2$.

The join of two metric spaces is defined in a similar way, but
before we do this we recall that if $(X,d)$ is a metric space, and
$t>0$, then $d^t$, defined by 
\begin{equation*}
d^t(x,y) = \min\{d(x,y),2t\},
\end{equation*}
is a metric on $X$. If $d(x,y)<2t$ then $d^t(x,y) = d(x,y)$, so that
the $d^t$-metric topology coincides with the $d$-metric topology on
$X$. As the metric $d^t$ will appear in our definition of the join,
we first show how the metric dimension of a single metric space varies
when we distort the metric from $d$ to $d^t$ as above. From now on, the $k$-metric dimension of $(X,d^t)$ will be denoted by $\dim_k^t(X)$.

\begin{theorem}\label{151018b}
Let $(X,d)$ be a metric space, and $k$ a positive integer, and
suppose that $0 < s < t$. Then
${\rm dim}^s_k(X) \geq {\rm dim}^t_k(X) \geq {\rm dim}_k(X)$.
However, it can happen that
\begin{equation}\label{160128b}
\lim_{t\to +\infty}\ {\rm dim}^t_k(X) > {\rm dim}_k(X).
\end{equation} 
\end{theorem}

The join of two metric spaces is defined in a similar way to the
join of two graphs, and to motivate this, suppose that $(X,d)$ is a
metric space, and that $X_1$ and $X_2$ are bounded subsets $X$ whose
distance apart is very large compared with their diameters. Then, in
some sense, we can approximate the metric space $(X_1\cup X_2,d)$ by
replacing all values $d(x_1,x_2)$, where $x_j\in X_j$, by $t$, where
$t$ is some sort of average of the values $d(x_1,x_2)$.  We shall
now define the join, so suppose that $(X_1,d_1)$ and $(X_2,d_2)$
are metric spaces, with $X_1\cap X_2 = \varnothing$, and $t>0$. Then
the \emph{join} of $(X_1,d_1)$ and $(X_2,d_2)$ (relative to the
parameter $t$) is the metric space $(X_1\cup X_2, d^t)$, where 
\begin{equation*}
d^t(u,v) = 
\begin{cases}
d^t_1(u,v) &\text{if \ $u,v\in X_1$}; \\
d^t_2(u,v) &\text{if \ $u,v\in X_2$};\\ 
t &\text{if $u\in X_i$ and $v\in X_j$, where $i\neq j$}.
\end{cases}
\end{equation*} 
As with graphs, $X_1+X_2$ always represents the metric space $(X_1\cup
X_2,d^t)$, where in this case $t$ will be understood from the context.

We might hope that the metric dimension is additive with respect to
the join, but unfortunately it is not. Let $X_1 = \{1,3\}$ and $X_2 =
\{2,4\}$, each with the Euclidean metric, and let $t=1$. Then 
$X_1\cup X_2 = \{1,2,3,4\}$ with the metric $d^1$, where $d^1(1,3)=d^1(2,4)=2$
and, for all other $x$ and $y$, $d^1(x,y)=1$. The bisectors in $X_1+X_2$
are $X_1$, $X_2$ and $\varnothing$, and from this we conclude that
${\rm dim}_1^1(X_1+X_2) = 3$. Obviously, ${\rm dim}_1(X_1) = {\rm
dim}_1(X_2) = 1$, so that in this case, 
${\rm dim}_1(X_1) + {\rm dim}_1(X_2) < {\rm dim}_1^1(X_1+X_2).$

We now give some inequalities which hold for the join of two metric
spaces. 

\begin{theorem}\label{160119a}
Let $(X_j,d_j)$, $j=1,2$, be metric spaces with $X_1\cap X_2 =
\varnothing$, and consider the join $(X_1\cup X_2,d^t)$. Then, for any
positive integer $k$, we have
\begin{equation}\label{151017b}
{\rm dim}_k(X_1) +  {\rm dim}_k(X_2) 
\leq {\rm dim}^t_k(X_1) +  {\rm dim}^t_k(X_2)
\leq {\rm dim}^t_k(X_1+X_2).
\end{equation} 
\end{theorem}

We shall now give an example which shows that \eqref{160128b} can
hold; then we end with the proofs of Theorems \ref{151018b} and
\ref{160119a}, and stating a consequence of Theorem \ref{160119a}.

\begin{example}{\rm 
Let $X=\mathbb{R}$ and $d(x,y)= |x-y|$, so that ${\rm dim}_1(X) = 2$. 
We shall now show that if $t>0$ then ${\rm dim}^t_1(X) = +\infty$,
so that \eqref{160128b} can hold.
Suppose that $a<b$, and consider the bisector $B^t(a|b)$. If $x \leq
a-2t$, then $d^t(x,a) = d^t(x,b) = 2t$ so that $x \in B^t(a|b)$. Thus
$B^t(a|b) \supset (-\infty,a-2t]$. Now let $S$ be any finite set,
and let $s$ be the largest element in $S$. Then 
$B^t(s+2t,s+3t) \supset (-\infty, s] \supset S$, so that
${\rm dim}^t_1(X) = +\infty$.
}\end{example}

This is a convenient place to describe the notation that will be
used in the following two proofs. We have metric spaces $(X_1,d_1)$
and $(X_2,d_2)$ with $X_1\cap X_2 = \varnothing$. For $j=1,2$ we use
$B_j(u|v)$ for the bisectors in $X_j$, and ${\rm dim}_k(X_j)$ for
their metric dimensions. Now consider the join $(X_1\cup X_2,d^t)$,
and its metric subspaces $(X_j,d^t)$. We use $B^t(u|v)$ and
$B^t_j(u|v)$ for the bisectors in these spaces, and ${\rm dim}^t_k(X_1
+ X_2)$ and ${\rm dim}^t_k(X_j)$ for their metric dimensions. In
general, we write $[B]^c$ for the complement of a bisector $B$ of any
type.

We shall need the following lemma in our proof of Theorem
\ref{151018b}.  

\begin{lemma}\label{151018c}
Let $(X,d)$ be a metric space, and suppose that $0 < s< t$. Then
$B(u|v) \subset B^t(u|v) \subset B^s(u|v)$.
\end{lemma}

\begin{proof}
First, observe that for all real $r$, and all real, distinct $\alpha$
and $\beta$, we have $\min\{\alpha,r\} = \min\{\beta,r\}$ if and only
if (i) $r \leq \min\{\alpha,\beta\}$ or (ii) $\alpha = \beta$. 
Now suppose that $x\in B^t(u|v)$. Then $d^t(x,u)=d^t(x,v)$ so that
$\min\{d(x,u),t \} = \min \{d(x,v),t \}$. This implies that 
$t \leq \min\{d(x,u),d(x,v)\}$ or $d(x,u)=d(x,v)$, and (since $s <
t$) in both cases we have $d^s(x,u) = d^s(x,v)$. Thus
$B^t(u|v) \subset B^s(u|v)$. The proof that $B(u|v) \subset B^t(u|v)$
is trivial: if $x\in B(u|v)$ then $d(x,u)=d(x,v)$ so that
$d^t(x,u)=d^t(x,v)$; hence $x\in B^t(u|v)$.
\end{proof}

\begin{proof}[The proof of Theorem \ref{151018b}] 
Let $A$ be any finite subset of $X$. Then, by Lemma \ref{151018c}, for
all $u$ and $v$ in $X$ with $u\neq v$, we have
$$|A \cap [B(u|v)]^c| \geq 
|A \cap [B^t(u|v)]^c| \geq |A \cap [B^s(u|v)]^c|.$$
It follows that if $A$ is a $k$-metric generator for $(X,d^s)$ 
(that is, if, for all $u$ and $v$, $|A \cap [B^s(u|v)]^c|\geq k$),
then it is also a $k$-metric generator for $(X,d^t)$. Thus the minimum
of $|S|$ taken over all $k$-metric generators $S$ of $(X,d^t)$ is 
less than or equal to the minimum over all $k$-metric generators of
$(X,d^s)$; hence ${\rm dim}^s_k(X) \geq {\rm dim}^t_k(X)$. The proof
that ${\rm dim}^t_k(X) \geq {\rm dim}_k(X)$ is entirely similar.
\end{proof}

\begin{proof}[The proof of Theorem \ref{160119a}]
 The first inequality follows from Theorem \ref{151018b}.
The inequality is trivially true if ${\rm dim}^t_k(X_1+X_2) =
+\infty$, so we may assume that there is a $k$-metric basis, say $W$,
of $X_1+X_2$. Thus $|W| = {\rm dim}^t_k(X_1+X_2)$.
Now take any $u$ and $v$ in $X_1$; then 
$$ B^t(u|v) = \{x \in X_1\cup X_2: d^t(x,u)=d^t(x,v)\} 
= B^t_1(u|v) \cup X_2, $$
so that, from Lemma \ref{151018c}, $[B^t(u|v)]^c = [B_1^t(u|v)]^c
\subset X_1$. 
We put $W_j=W \cap X_j$, $j=1,2$. Then, if we let $u$ and $v$
vary over $X_1$, with $u \neq v$, we find that
$$k \leq |[B^t(u|v)]^c \cap W| = |[B_1^t(u|v)]^c \cap X_1 \cap W|
= |[B_1^t(u|v)]^c \cap W_1|,$$ 
so that ${\rm dim}^t_k(X_1) \leq |W_1|$.
Similarly, ${\rm dim}^t_k(X_2) \leq |W_2|,$ so that
$${\rm dim}^t_k(X_1) + {\rm dim}^t_k(X_2) \leq |W_1|+|W_2| = |W|
= {\rm dim}^t_k(X_1 + X_2)$$
as required.
\end{proof}

If $(X_j,d_j)$, $j=1,2$, are metric spaces, each with diameter less
than $t$, such that $X_1\cap X_2= \varnothing$, the for any $k$-metric basis $A_i$ of $(X_j,d_j)$, $A_1\cup A_2$ is a $k$-metric generator for the join $(X_1\cup X_2,d^t)$. This shows that $\dim^t_k(X_1+X_2)\le \dim_k(X_1)+\dim_k(X_2)$, and so Theorem \ref{160119a} leads to the following corollary. 

\begin{corollary}\label{TrivialCaseJoin}
Let $(X_j,d_j)$, $j=1,2$, be metric spaces, each with diameter less
than $t$, such that $X_1\cap X_2= \varnothing$. Then, for
$k=1,2,\ldots$, $\dim^t_k(X_1+X_2)=\dim_k(X_1)+\dim_k(X_2)$.
\end{corollary}



\end{document}